\documentclass[12pt]{amsart}
\theoremstyle{plain}
\newtheorem{thm}{Theorem}[section]
\newtheorem{lem}[thm]{Lemma}
\newtheorem{cor}[thm]{Corollary}
\newtheorem{prop}[thm]{Proposition}
\newtheorem{ques}[thm]{Question}

\theoremstyle{remark}
\newtheorem{rem}[thm]{Remark}

\theoremstyle{definition}
\newtheorem{defi}[thm]{Definition}

\usepackage{amscd,amssymb,comment,epic,eepic,euscript,graphics}
\usepackage{enumerate}
\usepackage[initials]{amsrefs}

\usepackage{color}

\newcount\theTime
\newcount\theHour
\newcount\theMinute
\newcount\theMinuteTens
\newcount\theScratch
\theTime=\number\time
\theHour=\theTime
\divide\theHour by 60
\theScratch=\theHour
\multiply\theScratch by 60
\theMinute=\theTime
\advance\theMinute by -\theScratch
\theMinuteTens=\theMinute
\divide\theMinuteTens by 10
\theScratch=\theMinuteTens
\multiply\theScratch by 10
\advance\theMinute by -\theScratch

\def\today{{\number\day\space
 \ifcase\month\or
  January\or February\or March\or April\or May\or June\or
  July\or August\or September\or October\or November\or December\fi
 \space\number\year}}

\newcommand\Cpx{{\mathbb C}}
\newcommand\Dc{{\mathcal{D}}}

\newcommand\eps{\varepsilon}

\newcommand\HEu{{\EuScript H}}                   

\newcommand\Lc{{\mathcal{L}}}
\newcommand\Mcal{{\mathcal{M}}}
\newcommand\Nats{{\mathbb N}}

\newcommand\Reals{{\mathbb R}}
\newcommand\restrict{{\upharpoonright}}
\newcommand\Sc{{\mathcal{S}}}

\newcommand\Tcirc{{\mathbb T}}

\newcommand\Uc{{\mathcal{U}}}
\newcommand\Vc{{\mathcal{V}}}

\begin{document}

\title[The algebras $\Lc_{\log}$]{Algebras of log-integrable functions and operators}

\author[Dykema]{K. Dykema$^*$}
\address{Ken Dykema, Department of Mathematics, Texas A\&M University, College Station, TX, USA.}
\email{ken.dykema@math.tamu.edu}
\thanks{\footnotesize ${}^{*}$ Research supported in part by NSF grant DMS--1202660.}
\author[Sukochev]{F. Sukochev$^{\S}$}
\address{Fedor Sukochev, School of Mathematics and Statistics, University of new South Wales, Kensington, NSW, Australia.}
\email{f.sukochev@math.unsw.edu.au}
\thanks{\footnotesize ${}^{\S}$ Research supported by ARC}
\author[Zanin]{D. Zanin$^{\S}$}
\address{Dmitriy Zanin, School of Mathematics and Statistics, University of new South Wales, Kensington, NSW, Australia.}
\email{d.zanin@math.unsw.edu.au}



\begin{abstract}
We show that certain spaces of $\log$-integrable functions and operators are complete topological $*$-algebras with respect
to a natural metric space structure.
We explore connections with the Nevanlinna class of holomorphic functions.
\end{abstract}

\date{September 10, 2015}

\maketitle

\section{Introduction}

Let $(\Omega,\nu)$ be a measure space.
The symmetric function space $\Lc_{\log}(\Omega,\nu)$
consisting of measurable functions $f$ such that $\int\log(1+|f|)\,d\nu<\infty$,
arises naturally
as a commutative version of an algebra of operators important in von Neumann algebra theory.
Both this symmetric function space  and the (noncommutative) operator algebra version are closed under multiplication and, thus, are algebras.
In this note, we show that both are complete topological $*$-algebras, with respect to a metric space structure 
that arises from a naturally occuring F-norm, $\|f\|_{\log}:=\int\log(1+|f|)\,d\nu$.

In fact, we will see below that the space $\Lc_{\log}(\Omega,\nu)$ is a non-locally-convex (generalized) Orlicz space~\cite{MO61},
and the F-norm $\|\cdot\|_{\log}$ is equivalent to the one from~\cite{MO61}.
The completeness and other properties of $\Lc_{\log}$ that we prove will, thus, follow from more general results about such Orlicz
spaces.
However, we want to prove them using the F-norm $\|\cdot\|_{\log}$ because these are easier versions of the noncommutative case
(discussed below).

In the case that the measure space is the unit circle in the complex plane endowed with Lebesgue measure $m$,
boundary values of Nevanlinna functions are in the symmetric function space $\Lc(\Tcirc,m)$ and the map sending a Nevanlinna
function to its boundary values provides
an injective, continuous algebra homomorphism from the Nevanlinna class to $\Lc_{\log}(\Tcirc,m)$.
Since the Nevanlinna class is not well behaved under the usual metric, we propose studying the topological structure on the Nevanlinna
class provided by $\|\cdot\|_{\log}$.

If $(\Mcal,\tau)$ is a pair consisting of a von Neumann algebra $\Mcal$ and a normal, faithful, finite or semifinite trace $\tau$,
then $\Lc_{\log}(\Mcal,\tau)$ is defined to be the set of all $\tau$-measurable operators $T$ affiliated with $\Mcal$,
such that $\tau(\log(1+|T|))<\infty$.
In the case of a finite trace $\tau$,
this operator algebra version is important in that it is the natural domain of the Fuglede-Kadison
determinant~\cite{FK} and Brown measure~\cite{Brown}.
See~\cite{HS07} for a treatment of these topics in the setting of $\Lc_{\log}(\Mcal,\tau)$.

Our original motivation for investigating these algebras and proving that they are complete topological $*$-algebras
was for use in an effort to construct certain invariant subspaces and upper-triangular-type decompositions of unbounded
operators belonging to $\Lc_{\log}(\Mcal,\tau)$.
This was accomplished in~\cite{DSZ} using, in an integral way, the complete topological $*$-algebra structure of 
$\Lc_{\log}(\Mcal,\tau)$.

The contents of the rest of the paper are as follows.
Section~\ref{sec:comm} develops the theory of the symmetric function spaces (and algebras) $\Lc_{\log}(\Omega,\nu)$.
The brief Section~\ref{sec:Nev} makes the observations and asks questions about Nevanlinna class.
Section~\ref{sec:expL1} develops the theory of the (noncommutative) symmetric operator spaces (and algebras) $\Lc_{\log}(\Mcal,\tau)$.
We note that, although the commutative case $\Lc_{\log}(\Omega,\nu)$ is formally a special case
of symmetric operator spaces $\Lc_{\log}(\Mcal,\tau)$
and though the main results about $\Lc_{\log}(\Omega,\nu)$ are obtained in the greater generality of
the symmetric operator spaces $\Lc_{\log}(\Mcal,\tau)$,
we prefer to treat the commutative case separately, because the proofs are easier in this case, and because it may be
of some independent interest.

\medskip\noindent
{\bf Acknowledgements.}
The authors thank Laszlo Lempert, Mieczys{\l}aw Masty{\l}o and Alexis Poltoratski for helpful comments.

\section{Commutative $\Lc_{\log}$}
\label{sec:comm}

Fix a measure space $(\Omega,\nu)$ and let $\Lc_0=\Lc_0(\Omega,\nu)$ be the set of measurable, complex valued functions
on $\Omega$, 
with the usual convention that functions agreeing almost everywhere are identified as the same.
Let
\[
\Lc_{\log}(\Omega,\nu)=\left\{f\in\Lc_0\;\bigg|\;\int\log(1+|f|)\,d\nu<+\infty\right\}.
\]
If $\nu$ is finite, then we clearly have $\Lc_p(\Omega,\nu)\subseteq\Lc_{\log}(\Omega,\tau)$ for every $p\in(0,\infty)$.
It is easy to check that $\Lc_{\log}(\Omega,\nu)$ is a vector space.
For $f\in\Lc_{\log}(\Omega,\nu)$, define
\[
\|f\|_{\log}=\int(\log(1+|f|)\,d\nu.
\]
The next lemma shows that $\|\cdot\|_{\log}$ is an F-norm (see, e.g., \S1.2 of~\cite{KPR84}).
\begin{lem}\label{lem:Fnorm}
\begin{enumerate}[(a)]
\item\label{it:non0} $\|f\|_{\log}>0$ for all $f\ne0$.
\item\label{it:scale} $\|\alpha f\|_{\log}\le\|f\|_{\log}$ for all $f$ and all scalars $\alpha$ with $|\alpha|\le1$.
\item\label{it:lim0} $\lim_{\alpha\to0}\|\alpha f\|_{\log}=0$  for all $f$.
\item\label{it:subadd} $\|f+g\|_{\log}\le\|f\|_{\log}+\|g\|_{\log}$.
\end{enumerate}
\end{lem}
\begin{proof}
Parts \eqref{it:non0}--\eqref{it:lim0} are clear.
For part~\eqref{it:subadd}, use
\begin{align*}
\|f+g\|_{\log}&=\int\log(1+|f+g|)\,d\nu 
\le\int\log(1+|f|+|g|+|fg|)\,d\nu \\
&=\int\big(\log(1+|f|)+\log(1+|g|)\big)\,d\nu=\|f\|_{\log}+\|g\|_{\log}.
\end{align*}
\end{proof}

A consequence of Lemma~\ref{lem:Fnorm} is (see \S1.2 of~\cite{KPR84}) that
\begin{equation}\label{eq:metric}
d_{\log}(f,g):=\|f-g\|_{\log}
\end{equation}
is a translation invariant metric on $\Lc_{\log}(\Omega,\nu)$,
making $\Lc_{\log}(\Omega,\nu)$ into a topological vector space.
We henceforth regard $\Lc_{\log}(\Omega,\nu)$ as being endowed with this topology.

\begin{rem}\label{rem:Orlicz}
The function $\varphi(t)=\log(1+t)$ is a $\varphi$-function as defined in~\cite{MO61}.

We easily see that our space $\Lc_{\log}(\Omega,\nu)$ is, as a vector space, equal to the Orlicz space $\Lc^{*\varphi}(\Omega,\nu)$
which, in this case, equals $\Lc^{\varphi}(\Omega,\nu)$.
defined in \S2,1 of~\cite{MO61}.
It is also easy to see that the F-norm
\[
\|f\|_\varphi:=\inf\{\lambda>0\mid\left\|\frac{|f|}\lambda\right\|_{\log}\le\lambda\}
\]
on $\Lc_{\log}$
from~\cite{MO61}
(see, for example, Theorem 1.1 of~\cite{Ma89} for a proof that this is an F-norm)
is not equal to $\|f\|_{\log}$.
However, the next result implies that these two F-norms yield equivalent metrics and the same topology on $\Lc_{\log}$.
\end{rem}
\begin{lem}
Let $f\in\Lc_{\log}$.
\begin{enumerate}[(a)]
\item\label{it:logless} If $\|f\|_\varphi<1$, then $\|f\|_{\log}\le\|f\|_\varphi$.
\item\label{it:rhobound} If $N$ is an integer, then
\[
\|f\|_{\log}\le\frac1{N^2}\quad\implies\quad\|f\|_\varphi<\frac1N.
\]
\end{enumerate}
\end{lem}
\begin{proof}
If $|f|_\varphi<1$ then for every $\lambda<1$ such that $\|\frac f\lambda\|_{\log}\le\lambda$,
we have
\[
\|f\|_{\log}=\int\log(1+|f|)\,d\nu\le\int\log(1+\frac{|f|}\lambda)\,d\nu\le\lambda.
\]
Taking the infimum over such $\lambda$ proves~\eqref{it:logless}.

If $N$ is an integer and $\|f\|_{\log}\le\frac1{N^2}$ then by the subadditivity, Lemma~\ref{lem:Fnorm}\eqref{it:subadd},
we have
\[
\int\log(1+N|f|)\,d\nu=\|Nf\|_{\log}\le N\|f\|_{\log}<\frac1N,
\]
and this implies $\|f\|_\varphi<\frac1N$.
\end{proof}

The following result follows from Theorem 3.1 of~\cite{Ma89} and the equivalence of F-norms proved above.
However, here is a quick direct proof.
\begin{prop}
The space $\Lc_{\log}(\Omega,\nu)$ is a complete metric space.
\end{prop}
\begin{proof} 
This is a standard argument.
If $(f_n)_{n\ge1}$ is a Cauchy sequence in $\Lc_{\log}(\Omega,\nu)$, then this sequence is Cauchy in measure.
Hence, a subsequence of it converges almost everywhere to a measurable function $f$.
Now, using the Dominated Convergence Theorem, we can show that this subsequence
converges with respect to the metric~\eqref{eq:metric} to $f$
and, thus, the entire sequence $(f_n)_{n\ge1}$ converges to $f$.
\end{proof}

We now examine multiplication.

\begin{lem}\label{lem:mult}
Let $f,g,h\in\Lc_{\log}(\Omega,\nu)$ and let $K$ be a positive real number.
Then $fg\in\Lc_{\log}(\Omega,\nu)$ and
\begin{enumerate}[(a)]
\item\label{it:Kf} $\|Kf\|_{\log}\le\max(K,1)\|f\|_{\log}$,
\item\label{it:fg} $\|fg\|_{\log}\le\|f\|_{\log}+\|g\|_{\log}$,
\item\label{it:fgfh} if $|g|\le|h|$ almost everywhere, then $\|fg\|_{\log}\le\|fh\|_{\log}$.
\end{enumerate}
\end{lem}
\begin{proof}
For nonnegative real numbers $a$ and $b$, we have
\begin{align*}
\log(1+ab)&\le\max(b,1)\log(1+a) \\
\log(1+ab)&\le\log(1+a)+\log(1+b).
\end{align*}
The first implies~\eqref{it:Kf} and the second implies
\begin{multline*}
\|fg\|_{\log}=\int\log(1+|fg|)\,d\nu
\le\int\log(1+|f|+|g|+|fg|)\,d\nu \\
=\int\big(\log(1+|f|)+\log(1+|g|)\big)\,d\nu
=\|f\|_{\log}+\|g\|_{\log},
\end{multline*}
which proves~\eqref{it:fg}.
The assertion~\eqref{it:fgfh} is clear.
\end{proof}

\begin{prop}\label{prop:multconts}
Multiplication is continuous from $\Lc_{\log}(\Omega,\nu)\times\Lc_{\log}(\Omega,\nu)$ to $\Lc_{\log}(\Omega,\nu)$.
\end{prop}
\begin{proof}
We must show that, if sequences $\{f_n\}_{n\ge0}$ and $\{g_n\}_{n\ge0}$ converge in $\Lc_{\log}(\Omega,\nu)$ to $f$ and $g$, respectively,
then $\{f_ng_n\}_{n\ge0}$ converges to $fg$.
Using the additivity property we have
\[
\|f_ng_n-fg\|_{\log}\le\|(f_n-f)(g_n-g)\|_{\log}+\|f(g_n-g)\|_{\log}+\|(f_n-f)g\|_{\log}.
\]
By Lemma~\ref{lem:mult}, we have
\[
\|(f_n-f)(g_n-g)\|_{\log}\le\|f_n-f\|_{\log}+\|g_n-g\|_{\log}
\]
and this upper bound tends to $0$ as $n\to\infty$.
Let $K\ge1$.
Let $E=\{x\in\Omega\mid|f(x)|>K\}$.
Then using subadditivity of $\|\cdot\|_{\log}$ and Lemma~\ref{lem:mult}, we have
\begin{align*}
\|f(g_n-g)\|_{\log}&\le\|f(g_n-g)\chi_{E^c}\|_{\log}+\|f(g_n-g)\chi_E\|_{\log} \\
&\le\|K(g_n-g)\|_{\log}+\|f\chi_E\|_{\log}+\|g_n-g\|_{\log} \\
&\le(K+1)\|g_n-g\|_{\log}+\int_{\{|f|\ge K\}}\log(1+|f|)\,d\nu,
\end{align*}
where $\chi_E$ is the characteristic function of $E$ and $\chi_{E^c}=1-\chi_E$.
Thus,
\[
\limsup_{n\to\infty}\|f(g_n-g)\|_{\log}\le\int_{\{|f|\ge K\}}\log(1+|f|)\,d\nu.
\]
But the integral on the right-hand-side tends to $0$ as $K\to\infty$.
Therefore, we also have $\limsup_{n\to\infty}\|(f_n-f)g\|_{\log}=0$.
\end{proof}

Recall that a {\em topological algebra} is an algebra over $\Cpx$ or $\Reals$, endowed with a topology
such that the algebra operations (addition, scalar multiplication and multiplication) are jointly continuous.
We have, thus, proved:
\begin{cor}
$\Lc_{\log}(\Omega,\nu)$ is a topological algebra with respect to a complete metric space topology.
\end{cor}

The final conclusion of the following proposition also follows from Chapter~4 of~\cite{Ma89}
and the equivalence of F-norms --- see Remark~\ref{rem:Orlicz}.
however, we provide a proof below for our F-norm $\|\cdot\|_{\log}$;
note the parallel to Propostiion~\ref{prop:NCdense} and its proof.

\begin{prop}\label{prop:L1dense}
$\Lc_1(\Omega,\nu)$ is a dense subspace of $\Lc_{\log}(\Omega,\nu)$.
Furthermore, if $X\subseteq\Lc_1(\Omega,\nu)$ is dense with respect to the usual topology induced by $\|\cdot\|_1$,
then $X$ is dense in $\Lc_{\log}(\Omega,\nu)$ with respect to the topology induced by $\|\cdot\|_{\log}$.
Thus, if $\Lc_1(\Omega,\nu)$ is separable, then $\Lc_{\log}(\Omega,\nu)$ is separable.
\end{prop}
\begin{proof}
Since $\log(1+t)\le t$ for all $t\ge0$, we have $\Lc_1(\Omega,\nu)\subseteq\Lc_{\log}(\Omega,\nu)$
and
\begin{equation}\label{eq:logL1}
\|f\|_{\log}\le\|f\|_1
\end{equation}
for all $f\in\Lc_1(\Omega,\nu)$.
Given $f\in\Lc_{\log}(\Omega,\nu)$ and $M>0$, let 
\[
f_M(x)=\begin{cases}
f(x),&|f(x)|\le M \\
0,&\text{otherwise}.
\end{cases}
\]
For each fixed $M$, there exists $K_M>0$ such that $t\le K_M\log(1+t)$ for all $t\in[0,M]$
(in fact, we may take $K_M=M/\log(1+M)$).
Thus implies
$f_M\in\Lc_1(\Omega,\nu)$. 
But
\begin{equation}\label{eq:f-f_M}
\lim_{M\to\infty}\|f-f_M\|_{\log}=\int_{\{x\in\Omega\mid|f(x)|>M\}}\log(1+|f(x)|)\,d\nu(x)=0,
\end{equation}
by the Dominated Convergence Theorem.
This proves that $\Lc_1(\Omega,\nu)$ is dense in $\Lc_{\log}(\Omega,\nu)$.

Now suppose that 
$X$ is a dense subset in $\Lc_1(\Omega,\nu)$.
Let $f\in\Lc_{\log}(\Omega,\nu)$ and let $\eps>0$.
Using~\eqref{eq:f-f_M}, choose $M$ so that $\|f-f_M\|_{\log}<\eps$.
Now choose $g\in X$ so that $\|f_M-g\|_1<\eps$.
Then using~\eqref{eq:logL1}, we have $\|f-g\|_{\log}<2\eps$.
\end{proof}

Recall that a subset $B$ of a topological vector space is {\em bounded} if, for every neighborhood $\Vc$ of $0$,
there is $N\in\Nats$ such that $B\subseteq N\Vc$.
The next results shows that $\Lc_{\log}(\Omega,\nu)$ is not locally bounded.
(An essentially more general result appears in Theorem 5.1 of~\cite{Ma89}.)
It follows that the topology on $\Lc_{\log}(\Omega,\nu)$ is not determined by a quasi-norm.
\begin{prop}
If $\Omega$ has subsets of arbitrarily small positive measure, then
no neighborhood of $0$ in $\Lc_{\log}(\Omega,\nu)$ is bounded.
\end{prop}
\begin{proof}
A neighborhood base at $0$ for $\Lc_{\log}(\Omega,\nu)$ is $\{\Vc_\eps\mid\eps>0\}$, where
\[
\Vc_\eps=\{f\in\Lc_{\log}(\Omega,\nu)\mid \|f\|_{\log}<\eps\}.
\]
We will show that for every $\eps>0$, the set $\Vc_\eps$ fails to be bounded.
Let $N\in\Nats$.
We will show $\Vc_\eps\not\subset N\Vc_{\frac\eps2}$.
Take $f=K\chi_E$ for $K>0$ and for a measurable subset $E$ of $\Omega$ to be determined later.
Let $\eta=\nu(E)$.
We have
\[
\int\log(1+|f|)\,d\nu=\eta\log(1+K),\qquad\int\log(1+\frac{|f|}N)\,d\nu=\eta\log(1+\frac KN),
\]
so we want to choose $K$ and $\eta$ so that
$\log(1+|K|)<\frac\eps\eta$ but $\log(1+\frac KN)\ge\frac\eps{2\eta}$.
There exists $K_0>0$ such that for all $K\ge K_0$, we have
\[
2\log(1+\frac KN)>\log(1+K).
\]
There exists $\eta>0$ which is attained as the measure of some $E\subset\Omega$ and such that
\[
\log(1+K)<\frac\eps\eta\le2\log(1+\frac KN)
\]
for some $K\ge K_0$.
These fulfill the requirements.
\end{proof}

A result similar to the next one appears as Theorem 5.2 in~\cite{Ma89}.
\begin{prop}
Suppose $\nu$ is diffuse and $\sigma$-finite.
Then the space $\Lc_{\log}(\Omega,\nu)$ is not locally convex.
In fact, the only open, nonempty, convex subset is $\Lc_{\log}(\Omega,\nu)$ itself.
\end{prop}
\begin{proof}
The following argument is analogous to \S1.47 of~\cite{Ru73}.
We may without loss of generality assume $\Omega=[0,1]$.
Suppose $\Uc\subseteq\Lc_{\log}(\Omega,\nu)$ is open, convex and nonempty.
We will show $\Uc=\Lc_{\log}(\Omega,\nu)$.
Without loss of generality $0\in\Uc$ and, therefore, $\Vc_\eps\subseteq\Uc$ for some $\eps>0$.
Let $f\in\Lc_{\log}(\Omega,\nu)$.
Using the Dominated Convegence Theorem, we find
\[
\lim_{n\to\infty}\frac{\|nf\|_{\log}}n=\lim_{n\to\infty}\int\frac{\log(1+n|f|)}n\,d\nu=0.
\]
Indeed, all integers $n\ge1$ and all $x\ge0$, the inequality
\[
\log(1+nx)\le n\log(1+x)
\]
holds.
Thus for all $n\ge1$ the function $g_n=\frac{\log(1+n|f|)}n$ is dominiated by $\log(1+|f|)$, which is assumed to
be integrable, while $g_n$ tends pointwise to $0$ as $n\to\infty$.
Choose $n$ so large that $\frac{\|nf\|_{\log}}n<\eps$.
By continuity of the antiderivative, there exist $0=x_0<x_1<\cdots<x_n=1$ so that for all $j$,
\[
\int_{[x_{j-1},x_j]}\log(1+n|f|)\,d\nu=\frac{\|nf\|_{\log}}n.
\]
Let $f_j=nf\,\chi_{[x_{j-1},x_j]}$.
Then $f_j\in\Vc_\eps\subseteq\Uc$ for all $j$.
Since $\Uc$ is convex, we have
\[
f=\frac1n(f_1+\cdots+f_n)\in\Uc.
\]
So $\Uc=\Lc_{\log}(\Omega,\nu)$.
\end{proof}

\section{Relation to the Nevanlinna class}
\label{sec:Nev}

The Nevanlinna class $N$
is the set of functions $f$, holomorphic on the open unit disk, such that
\[
L(f):=\sup_{r<1}\frac1{2\pi}\int\log(1+|f(re^{i\theta})|)\,d\theta<\infty.
\]
By a theorem of F.\ and M.\ Reisz (see, e.g., Theorem 2.1 of~\cite{Du70}) $f\in N$ if and only if
$f$ is the ratio of two bounded analytic functions on the unit disk.

The function
\[
L(r,f)=\frac1{2\pi}\int\log(1+|f(re^{i\theta})|)\,d\theta
\]
is increasing in $r$, because $\log(1+|f|)$ is subharmonic.
In~\cite{SS76}, Shapiro and Shields investigated the metric space structure of $N$ with respect to the translation-invariant
metric $d_L(f,g)=L(f-g)$.
They showed that $N$ with this topology is not a topological vector space, is disconnected and, in fact,
has many linear subspaces that inherit the discrete topology.

For $f\in N$, (see Theorem~2.2 of~\cite{Du70}), nontangential limits $f(e^{i\theta}):=\lim_{r\to1^-}f(re^{i\theta})$
of $f$ exist almost everywhere at the boundary,
and the function $\theta\mapsto\log|f(e^{i\theta})|$ is integrable (except when $f$ is identically $0$).
In particular, the boundary function always belongs to $\Lc_{\log}(\Tcirc,m)$,
where $m$ is Haar measure on the unit circle,
and the map
\[
\Phi:N\to\Lc_{\log}(\Tcirc,m),
\]
that sends $f\in N$ to its boundary value function, is an injective algebra homomorphism.

By Fatou's Lemma,
\[
L(f)=\lim_{r\to1^-}\frac1{2\pi}\int_0^{2\pi}\log(1+|f(re^{i\theta})|)\,d\theta
\ge\frac1{2\pi}\int_0^{2\pi}\log(1+|f(e^{i\theta})|)\,d\theta=\|\Phi(f)\|_{\log},
\]
but equality need not hold.
In fact (see~ Proposition 1.2 of~\cite{SS76}), equality holds if and only if $f$ belongs to the Smirnov class $N^+\subset N$.

\begin{defi}
Consider the F--norm on class $N$ defined by $\|f\|_N=\|\Phi(f)\|_{\log}$.
The corresponding distance $d_N(f,g)=\|f-g\|_N$
makes $N$ into a topological algebra.
\end{defi}

\begin{ques}\label{ques:Ncomplete}
Is $N$ complete with respect to $d_N$?
Equivalently, is $\Phi(N)$ closed in $\Lc_{\log}(\Tcirc,m)$?
\end{ques}

The topological algebra
structure induced on $N$ by the metric $d_N$ seems better behaved than the topology induced by $d_L$.

\section{Noncommutative $\Lc_{\log}$ spaces}
\label{sec:expL1}

In this section we consider noncommutative analogues of the commutative $\Lc_{\log}$ spaces that were introduced in Section~\ref{sec:comm}.
Let $\Mcal\subseteq B(\HEu)$ be a von Neumann algebra possessing a normal, faithful, finite or semifinite trace $\tau$.
If $\tau$ is finite, we will assume it is normalized, namely, $\tau(1)=1$.
We will assume $\Mcal$ is diffuse, meaning that it has no minimal nonzero projections.
We let $\Lc_{\log}(\Mcal,\tau)$ be the set of all (possibly unbounded)
linear closed operators $T$ with dense domains
in $\HEu$ that are affiliated with $\Mcal$
and satisfy 
\[
\|T\|_{\log}:=\tau(\log(1+|T|))<\infty.
\]
This is a vector subspace of the set $\Sc(\Mcal,\tau)$ of all $\tau$-measurable operators $T$ affiliated with $\Mcal$, and an $\Mcal$-bimodule.

For $T\in\Sc(\Mcal,\tau)$, we let $\mu(T)$ be the generalized singular number function $x\mapsto\mu_x(T)$ for $T$ (see~\cite{FackKosaki}).
It is a nonincreasing, right continuous function from $(0,1]$ to $[0,\infty)$ if $\tau$ is finite, and from $(0,\infty)$ to $[0,\infty)$
otherwise.
Thus,
\begin{equation}\label{eq:Tlogint}
\|T\|_{\log}=\int\log(1+\mu(T))\,d\lambda,
\end{equation}
where $\lambda$ denotes Lebesgue measure and the above integral means $\int_0^1\log(1+\mu_x(T))\,dx$ if $\tau$ is finite
and $\int_0^\infty\log(1+\mu_x(T))\,dx$ if $\tau$ is infinite.

The next lemma shows that $\|\cdot\|_{\log}$ is an F-norm on $\Lc_{\log}(\Mcal,\tau)$, (see \S1.2 of~\cite{KPR84}).
\begin{lem}\label{lem:ncFnorm}
Let $S,T\in\Sc(\Mcal,\tau)$.
Then
\begin{enumerate}[(a)]
\item\label{it:Dnon0} $\|T\|_{\log}>0$ provided $T\ne0$.
\item\label{it:T*} $\|T^*\|_{\log}=\|T\|_{\log}$.
\item\label{it:Dscale} $\|\alpha T\|_{\log}\le\|T\|_{\log}$ for all scalars $\alpha$ with $|\alpha|\le1$.
\item\label{it:Dlim0}  If $T\in\Lc_{\log}(\Mcal,\tau)$, then $\lim_{\alpha\to0}\|\alpha T\|_{\log}=0$.
\item\label{it:Dsubadd} $\|S+T\|_{\log}\le\|S\|_{\log}+\|T\|_{\log}$.
\end{enumerate}
\end{lem}
\begin{proof}
Using $\mu(\alpha T)=\alpha\mu(T)$ for all $\alpha>0$ and
$\mu(T^*)=\mu(T)$,
we easily prove \eqref{it:Dnon0}-\eqref{it:Dlim0}.

The triangle inequality~\eqref{it:Dsubadd} follows from Theorem 4.7(i) of~\cite{FackKosaki},
with $g(x)=\log(1+x)$, where we use~\eqref{eq:Tlogint}.
\end{proof}

Since $\|\cdot\|_{\log}$ is an F-norm on $\Lc_{\log}(\Mcal,\tau)$,
it provides a translation invariant metric
\[ 
d_{\log}(S,T)=\|S-T\|_{\log}
\] 
making it into a topological vector space
(see Section 1.2 of~\cite{KPR84}).
Henceforth, we regard $\Lc_{\log}(\Mcal,\tau)$ as endowed with this topology.

\begin{rem}\label{rem:D}
If $\Dc$ is an abelian von Neumann subalgebra of $\Mcal$, and $\Dc\cong L^\infty(\Omega,\nu)$
for a measure space $(\Omega,\nu)$, where $\tau\restrict_\Dc$ is given by integration against $\nu$,
then space consisting those $T\in\Lc_{\log}(\Mcal,\tau)$ that are affiliated to $\Dc$
is naturally identified with $\Lc_{\log}(\Omega,\nu)$ and the two definitions of $\|\cdot\|_{\log}$ coincide.
\end{rem}

We now examine multiplication in $\Lc_{\log}(\Mcal,\tau)$.

\begin{lem}\label{lem:NCmult}
Let $S,T\in\Lc_{\log}(\Mcal,\tau)$ and let $K$ be a positive real number.
Then $ST\in\Lc_{\log}(\Mcal,\tau)$ and
\begin{enumerate}[(a)]
\item\label{it:ST} $\|ST\|_{\log}\le\|S\|_{\log}+\|T\|_{\log}$.
\item\label{it:KS} Furthermore, if $S$ is bounded, then $\|ST\|_{\log}\le\max(\|S\|,1)\|T\|_{\log}$.
\end{enumerate}
\end{lem}
\begin{proof}
For~\eqref{it:ST}, we use Theorem 4.2(iii) of~\cite{FackKosaki} with the function $f(x)=\log(1+x)$
(since we easily see that $t\mapsto\log(1+e^t)$ is convex on $\Reals$).
Thus, we get
\[
\|ST\|_{\log}=\int \log(1+\mu(ST))\,d\lambda\le
\int \log(1+\mu(S)\mu(T))\,d\lambda\le\|S\|_{\log}+\|T\|_{\log},
\]
where the last estimate follows as in the proof of Lemma~\ref{lem:mult}\eqref{it:fg}.
This proves~\eqref{it:ST}

To prove~\eqref{it:KS},
if $S$ is bounded, then we have $\mu(ST_n)\le\|S\|\mu(T_n)$.
Thus,
\begin{multline*}
\|ST\|_{\log}\le\int\log(1+\|S\|\mu(T))\,d\lambda\le\max(\|S\|,1)\int\log(1+\mu(T))\,d\lambda \\
=\max(\|S\|,1)\|T\|_{\log}.
\end{multline*}
\end{proof}

\begin{lem}\label{lem:onemultconts}
Let $S\in\Lc_{\log}(\Mcal,\tau)$ and suppose $T_n$ is a sequence in $\Lc_{\log}(\Mcal,\tau)$ that converges to $0$.
Then $ST_n\to0$ and $T_nS\to0$.
\end{lem}
\begin{proof}
Since $\|T_nS\|_{\log}=\|S^*T_n^*\|_{\log}$, we need only show $\|ST_n\|_{\log}\to0$.
If $S$ is bounded, then this follows from Lemma~\ref{lem:NCmult}.
Now for general $S\in\Lc_{\log}(\Mcal,\tau)$, take $0<K<\infty$ and let $S_K=SE_{|S|}((K,\infty))$,
where $E_{|S|}$ denotes the projection-valued spectral measure of $|S|$.
Then $S-S_K=SE_{|S|}([0,K])$ is bounded.
Using Lemmas~\ref{lem:ncFnorm} and~\ref{lem:NCmult}, we have
\begin{multline*}
\limsup_{n\to\infty}\|ST_n\|_{\log}\le\limsup_{n\to\infty}\|S_KT_n\|_{\log}+\|(S-S_K)T_n\|_{\log} \\
=\limsup_{n\to\infty}\|S_KT_n\|_{\log}
\le\limsup_{n\to\infty}(\|S_K\|_{\log}+\|T_n\|_{\log})=\|S_K\|_{\log}.
\end{multline*}
But
\[
\|S_K\|_{\log}=\int_{\{x\mid\mu_x(S)>K\}}\log(1+\mu(S))\,d\lambda
\]
tends to $0$ as $K\to\infty$.
Thus, $\|ST_n\|\to0$ as $n\to\infty$.
\end{proof}

\begin{prop}\label{prop:NCmultconts}
Multiplication is continuous from $\Lc_{\log}(\Mcal,\tau)\times\Lc_{\log}(\Mcal,\tau)$ to $\Lc_{\log}(\Mcal,\tau)$.
\end{prop}
\begin{proof}
Suppose $S_n$ and $T_n$ are sequences in $\Lc_{\log}(\Mcal,\tau)$ and $S_n\to S$ and $T_n\to T$.
We must show $S_nT_n\to ST$.
But using Lemmas~\ref{lem:ncFnorm} and~\ref{lem:NCmult}, we have
\begin{multline*}
\|S_nT_n-ST\|_{\log}
\le\|(S_n-S)(T_n-T)\|_{\log}+\|S(T_n-T)\|_{\log}+\|(S_n-S)T\|_{\log} \\
\le\|(S_n-S)\|_{\log}+\|(T_n-T)\|_{\log}+\|S(T_n-T)\|_{\log}+\|(S_n-S)T\|_{\log}.
\end{multline*}
From Lemma~\ref{lem:onemultconts}, we conclude
$\lim_{n\to\infty}\|S_nT_n-ST\|_{\log}=0$.
\end{proof}

\begin{cor}\label{cor:top*alg}
$\Lc_{\log}(\Mcal,\tau)$ is a topological $*$-algebra.
\end{cor}
\begin{proof}
This is an immediate consequence of Lemma~\ref{lem:ncFnorm}, Proposition~\ref{prop:NCmultconts} and
the obvious fact that $\|T^*\|_{\log}=\|T\|_{\log}$.
\end{proof}

\begin{prop}\label{prop:NCdense}
$\Lc_1(\Mcal,\tau)$ is dense in $\Lc_{\log}(\Mcal,\tau)$.
If $X$ is any subset of $\Lc_1(\Mcal,\tau)$
that is dense in $\Lc_1(\Mcal,\tau)$ with respect to the usual norm $\|\cdot\|_1$, then $X$ is dense in $\Lc_{\log}(\Mcal,\tau)$.
Consequently, if $\Mcal$ has separable predual, then $\Lc_{\log}(\Mcal,\tau)$ is separable.
\end{prop}
\begin{proof}
Let $T\in\Lc_{\log}(\Mcal,\tau)$.
Let $T=U|T|$ be the polar decomposition of $T$.
By Lemma~\ref{lem:NCmult}, 
Using Remark~\ref{rem:D} and Proposition~\ref{prop:L1dense}, (or using the proof of the latter)
we find that $\||T|-R\|_{\log}$ can be made arbitrarly small for $R\in\Lc_1(\Mcal,\tau)$.
Using Lemma~\ref{lem:NCmult}\eqref{it:KS}, we have $\|T-UR\|_{\log}\le\||T|-R\|_{\log}$, and $UR\in\Lc_1(\Mcal,\tau)$.
This shows that $\Lc_1(\Mcal,\tau)$ is dense in $\Lc_{\log}(\Mcal,\tau)$.

We have also here $\|T\|_{\log}\le\|T\|_1$, so density of $X$ on $\Lc_1(\Mcal,\tau)$ implies density in $\Lc_{\log}(\Mcal,\tau)$.

As is well known,
(see, for example, Appendix B.5 of~\cite{SiSm}),
the predual of $\Mcal$ is $\Lc^1(\Mcal,\tau)$,
from which the last statement of the lemma follows.
\end{proof}

Recall that, on $\Sc(\Mcal,\tau)$,
the measure topology of E.\ Nelson~\cite{Ne74}
is translation invariant and has as neighborhood base at $0$ the set $\{N_{\eta,\delta}\mid \eta,\delta>0\}$,
where
\[
N_{\eta,\delta}=\{A\in\Sc(\Mcal,\tau)\mid \tau(E_{|A|}([\delta,\infty)))<\eta\}.
\]
This topology is metrizable
for example, by the metric
\begin{equation}\label{eq:dtau}
d_\tau(A,B)=\sum_{n=1}^\infty 2^{-n}\tau(E_{|A-B|}([n^{-1},\infty))).
\end{equation}
and a sequence $\{A_n\}_{n\ge0}$ in $\Sc(\Mcal,\tau)$ converges in measure to $A$
if and only if
$$\lim_{n\to\infty}\tau(E_{|A_n-A|}([\delta,\infty)))=0 ,\qquad(\forall \delta>0).$$
When $\Mcal$ can be represented on a separable Hilbert space,
it is also separable.
Indeed, it is a classical result that in this case, the unit ball of the algebra $B(\mathcal{H})$ is separable in the $(so^*)$-topology and therefore, the unit ball of the algebra $\Mcal$  is also separable in that topology. It is easy to see that the latter topology reduced to the set of all projections $\mathcal{P}(\Mcal)$ coincides with the measure topology on that set. Finally, it is easy to see that the space $\Sc(\Mcal,\tau)$ is separable in the measure topology if and only if the set $\mathcal{P}(\Mcal)$ is separable in the measure topology.
For further details, we refer to~\cite{S85}.
Furthermore, the set of positive elements is closed in this topology (see~\cite{DDP}).

\begin{rem}\label{rem:as strong}
Using
\begin{equation}\label{eq:logchar}
\tau(E_{|A|}([\delta,\infty)))\le\tau(1+3\delta^{-1}|A|)=\|3\delta^{-1}A\|_{\log},
\end{equation}
we immediately see that
the $\|\cdot\|_{\log}$-topology on $\Lc_{\log}(\Mcal,\tau)$
is at least as strong as the relative measure topology on $\Lc_{\log}(\Mcal,\tau)$
(and it is not difficult to show it is strictly stronger).
Thus, we have that
\[
\{X\in\Lc_{\log}(\Mcal,\tau)\mid X\ge0\}
\]
is closed with respect to $\|\cdot\|_{\log}$.
\end{rem}

In the remainder of this paper, we show that $\Lc_{\log}(\Mcal,\tau)$ is complete,
as an application of Nelson's analogous result~\cite{Ne74} for the notion of Cauchy in measure.
\begin{thm}\label{logl1 topol alg}
$\Lc_{\log}(\mathcal{M},\tau)$ is a topological $*-$algebra with respect to a complete metric space topology.
\end{thm}
\begin{proof}
In light of Corollary~\ref{cor:top*alg}, we need only show that the metric $d_{\log}$ is complete.
Suppose a sequence $\{A_n\}_{n\ge0}$ is Cauchy in $\|\cdot\|_{\log}$.
Nelson's definition is that a sequence $\{A_n\}_{n\ge0}$ in $\Sc(\Mcal,\tau)$ is Cauchy in measure if, for all $\delta>0$, the quantity
$\tau(E_{|A_n-A_m|}([\delta,\infty)))$ tends to $0$ as $n,m\to\infty$, and this is equivalent to being Cauchy with respect to the metric
$d_\tau$ from~\eqref{eq:dtau}.
Using the inequality~\eqref{eq:logchar},
we see immediately that if a sequence is Cauchy with respect to $\|\cdot\|_{\log}$, then it is Cauchy in measure, and
Nelson showed (see the end of Section~2 of~\cite{Ne74}) that this implies that the sequence converges in measure.
Let $A$ be the limit.

We claim that $A\in\Lc_{\log}(\mathcal{M},\tau)$.
Using Lemma~\ref{lem:ncFnorm}\eqref{it:Dsubadd}, we see that $\|A_n\|_{\log}$ stays bounded as $n\to\infty$.
It follows from
Fack and Kosaki's version of Fatou's Lemma in $\Sc(\Mcal,\tau)$, namely, Theorem~3.5(i) of~\cite{FackKosaki},
that $\log(1+|A|)\in\Lc_1(\mathcal{M},\tau)$.
This proves the claim.

It remains to show that $\|A_n-A\|_{\log}\to0$.
Given $\eps>0$, there exists an integer $N$ such that
$$\tau(\log(1+|A_n-A_m|))\leq\varepsilon,\quad(n,m>N).$$
Letting $m\to\infty$ and again using Fack and Kosaki's version of Fatou's Lemma,
we have
$$\|A_n-A\|_{\log}=\tau(\log(1+|A_n-A|))\leq\eps,\quad (n>N).$$
\end{proof}

\end{document}